\documentclass[preprint,12pt,authoryear]{article}
\usepackage{lineno,hyperref}
\usepackage{hyperref}

\usepackage{amssymb,amsmath,amsthm}

\newtheorem{teor}{Theorem}

\newtheorem{lema}{Lemma}

\newtheorem{obs}{Remark}

\newcommand{\vel}{\mathbf u}
\newcommand{\veld}{\mathbf u^\delta}
\newcommand{\R}{\mathbb R}
\newcommand{\x}{\mathbf x}
\newcommand{\y}{\mathbf y}

\newcommand{\z}{\mathbf z}

\newcommand{\dosnorma}[1]{\|#1\|_{L^ 2(\mathbb R^2)}}
\newcommand{\xnorma}[2]{\|#1\|_{L^ {#2}(\mathbb R^2)}}
\newcommand{\pnorma}[1]{\|#1\|_{L^ p(\mathbb R^2)}}
\newcommand{\inorma}[1]{\|#1\|_{L^ \infty(\mathbb R^2)}}

\newcommand{\llseminorma}[1]{\langle#1\rangle_{LL}}
\newcommand{\llnorma}[1]{\|#1\|_{LL}}

\newcommand{\diver}{\mathrm {div\ }}

\newcommand{\essinf}{\mathrm {ess inf\ }}

\newcommand{\supintnormard}[3]{\sup\limits_{0\le \tau\le #1}\int |#2|^2d\x+\int_0^{#1}\int \rho|#3|^2d\x d\tau}

\newcommand{\supintnormardsobconv}[4]{\sup\limits_{0\le \tau\le #1}\int \tau^{#4}\rho|#2|^2d\x+\int_0^{#1}\int \tau^{#4}|#3|^2d\x d\tau}

\providecommand{\keywords}[1]{\textbf{\textit{Keywords:}} #1}

\begin{document}
\title{Lagrangian structure for two dimensional non-barotropic compressible fluids\thanks{Financial support by Fapesp, grant 2009/15515-0.}}
\author{Pedro Nel Maluendas Pardo\\
School of Mathematics and Statistics\\ 
UPTC--Universidad Pedag\'ogica y Tecnol\'ogica de Colombia\\
 Avenida Central del Norte 39-115. 150003\\ 
Tunja, Boyac\'a, Colombia\\ 
pedro.maluendas@uptc.edu.co\\
\\ 
Marcelo M. Santos\\
IMECC--Institute of Mathematics, Statistics and Scientific Computing\\ 
UNICAMP--Universidade Estadual de Campinas.\\ 
Rua S\'ergio Buarque de Holanda, 651. 13083-859\\ 
Campinas, SP, Brazil\\ 
msantos@ime.unicamp.br}

\date{}

\maketitle

\begin{abstract}
We study the \emph{lagrangian structure} for weak solutions of two dimensional Navier-Stokes equations for a non-barotropic compressible fluid, i.e. we show the uniqueness of particle trajectories for two dimensional compressible fluids including the energy equation.
Our result extends partially the previous result obtained for barotropic fluids by D. Hoff and M. M. Santos \cite{hoff-santos}.
\end{abstract}
%


\maketitle

\section{Introduction}

\label{introduction}

We consider the model equations for the non-barotropic (and polytropic) compressible fluids
\begin{align}
\label{masaeq}&\rho_t + \diver\left( \rho \vel \right)=0\\
\label{momentoeq}&\left(\rho\vel \right)_t + \diver \left( \rho \vel \otimes \vel \right) + \nabla P = \mu \triangle \vel + \lambda \nabla\left( \diver \vel \right)\\
\label{energiaeq}
&\left( \rho e  \right)_t + \diver \left(\rho e \vel \right)
= K \triangle e - P \diver \vel + \mu \left(|\nabla \vel |^2 
+ u^k_{x_j}u^j_{x_k} \right) 
+ \left( \lambda - \mu  \right)\left( \diver \vel  \right)^2,
\end{align}
with $t>0$ and $x=(x_1,x_2)\in\mathbb{R}^2$, where repeated indexes mean summation from 
$1$ to $2$, and $\rho$, $\vel=(u^1,u^2)$, $e$, $P$ and $\mu, \lambda>0$ (constants) denote, respectively, the density, velocity, specific internal energy, pressure and viscosities of the fluid. We assume that the fluid is ideal, i.e. $P(\rho,e)=(\gamma-1)\rho e$, where $\gamma>1$ is a constant (the adiabatic constant).  The symbol $K$ denotes some positive constant related to the heat flow. These three equations describe, respectively, the conservation of mass, the conservation of momentum and the balance of energy (see e.g. \cite{feireisl}, \cite{batchelor} or \cite{anderson}).

In terms of the convective derivative \ $\mathbf{\dot{}}  := \partial_t + \vel \cdot \nabla$, \ assuming that $(\rho,\vel,e)$ is sufficiently regular and using the equation \eqref{masaeq} in the the equations \eqref{momentoeq} and
\eqref{energiaeq}, we can write the system \eqref{masaeq}-\eqref{energiaeq} as
\begin{align}
\label{masaeqc}&\dot \rho = -\rho \diver \vel\\
\label{momentoeqc}&\rho \dot \vel = -\nabla P + \mu\triangle \vel + \lambda \nabla(\diver \vel)\\
\label{energiaeqc}&\rho \dot e = K \triangle e - P \diver \vel + \mu \left(|\nabla
\vel |^2 + u^k_{x_j}u^j_{x_k} \right) +
    \left( \lambda - \mu  \right)\left( \diver \vel  \right)^2.
\end{align}
To the system \eqref{masaeq}-\eqref{energiaeq} we add the initial conditions
\begin{equation}
\label{ci}
\rho(0,\cdot) = \rho_0, \ \ \vel(0,\cdot) = \vel_0, \ \ e(0,\cdot)=e_0,
\end{equation}
which can be discontinuous functions. We shall assume that $\vel_0$ belongs to the Sobolev space $H^1(\R^2)$, and, for constants $\tilde\rho, \tilde e$, and $l>0$, that the initial {\lq\lq}energy{\rq\rq} 
\begin{equation}\label{E0}
\begin{array}{lrl}
&C_0:=&\displaystyle{\inorma{\rho_0-\tilde\rho}^2+\|\vel_0\|_{H^1(\R^2)}^2}\\
&
&+\displaystyle{\int\left[(\rho_0-\tilde\rho)^2+|\vel_0|^2+|e_0-\tilde
e|^2+|\nabla e_0|^2\right](1+|\x|^2)^ld\x},
\end{array}
\end{equation}
is sufficiently small. 

Our goal in this paper is to show the uniqueness of particle paths
(trajectories of the velocity field $\vel$) of a weak solution $(\rho,\vel,e)$
of the system \eqref{masaeq}-\eqref{energiaeq} together with initial condition
\eqref{ci}, following the plan of \cite{hoff-santos} for barotropic fluids.
Briefly, a key idea in \cite{hoff-santos} is to write $\vel=\vel_{P}+\vel_{F,\omega}$, where $\vel_{P}$ is a vector field associated with the pressure $P$ and $\vel_{F,\omega}$ is associated with the vorticity $\omega$ and the so called {\em effective viscous flux}, i.e. the quantity $F := (\mu + \lambda)\mbox{div}\vel - (P-\tilde P)$, where $\tilde P:=P(\tilde\rho)$. By energy estimates, some properties of regularity of the solution and some estimates on the convective derivative of $\vel$ and using classical arguments of elliptic equations, it is shown that $\vel_{P}$ is a log-lipschitzian vector field in space, for each positive time. In addition, the log-lipschitzian seminorm of $\vel_P$ is locally bounded with respect to time.   On the other hand, by classical Sobolev estimates and also by some estimates on the convective derivative of $\vel$, it is possible to show that the vector field $\vel_{F,\omega}$ is lipschitzian in space, also for each positive time. Then, assuming that the initial velocity is in the Sobolev space $H^s(\mathbb{R}^2)$, for some arbitrary $s>0$, it is shown that the lipschitzian seminorm of $\vel_{F,\omega}$ is locally integrable with respect to time. This is perhaps the most difficult part. Putting together the results for $\vel_P$ and $\vel_{F,\omega}$, one has that the log-lipschitzian seminorm of the velocity field $\vel$ is locally integrable with respect to time. Therefore, the uniqueness of particle paths follows from Osgood's lemma.    
%
 We shall show that this procedure is applicable to the non-barotropic case \eqref{masaeq}-\eqref{energiaeq}, under the hypothesis that the initial velocity $\vel_0$ is in the Sobolev space $H^1$ with sufficiently small norm $\|\vel_0\|_{H^1}$. The nonlinearity $P(\rho,e)$ turns the problem  very difficult. We shall use the solution to the system \eqref{masaeq}-\eqref{energiaeq} obtained by \cite{hoff-heatconducting}. 

More precisely, in this paper we show that under the above additional hypothesis ($\|\vel_0\|_{H^1(\mathbb{R}^2)}<<1$), a similar theorem to Theorem 2.5 of \cite{hoff-santos} holds true for the equations \eqref{masaeq}-\eqref{energiaeq}, i.e. we shall prove the following result:

\begin{teor}
\label{lagrangiana}
Let $(\rho,\vel,e)$ be a weak solution of the system \eqref{masaeq}-\eqref{energiaeq} and initial conditions \eqref{ci}, as in \cite[Theorem 1.1]{hoff-heatconducting}. There is a positive number $\varepsilon$ such that if $E_0<\varepsilon$ then  
\begin{enumerate}
\item for each $\x_0\in \R^2$ there exits a unique map $X(\cdot, \x_0)\in
C([0,\infty];\R^2)\cap C^1((0,\infty);\R^2)$ satisfying
\begin{equation}\label{trajetoria}
X(t,\x_0)=\x_0+\int_0^t\vel(X(\tau,\x_0),\tau)d\tau, \quad\forall t\ge0;
\end{equation}
\item \label{homeomorfismo} for each $t>0$, the flux map $\x\in\mathbb{R}^2\mapsto X(t,\x)\in\mathbb{R}^2$ is a homeomorphism; 
\item \label{holdercompactos}for each compact set $K$ in $\R^2$ and any $0\le t_1<t_2<\infty$, the map $X(t_1,x)\mapsto X(t_2,x)$, $x\in K$, is bijective and H\"older 
continuous;
\item \label{curves H cont} for each $t>0$, the map $x\in\R^2\mapsto X(t,x)$ takes H\"older continuous curves into H\"older continuous curves, i.e. if  $\mathcal C$ 
is a curve of class $C^\alpha$ in $\R^2$, for some $\alpha\in [0,1)$ 
then $X(t,\mathcal C)$ is a curve of classe $C^{\alpha e^{-Lt}}$, where $L$ is a positive constant depending on 
$\underline \rho$ and $s$.
\end{enumerate}
\end{teor}

To the best of our knowledge, up to now only a few results are established regarding the lagrangian structure for compressible fluids. For barotropic fluids we can mention the following papers which prove the lagrangian structure: \cite{hoff-2dim}, in dimension two with the initial velocity in the Sobolev space $H^s$, for an arbitrary $s>0$, and with a piecewise H\"older continuous initial density across a single $C^{1,\alpha}$ curve; \ \cite{hoff-santos}, in dimension two and three, with the initial velocity in $H^s$, with $s>0$ in dimension two and $s>1/2$ in dimension three; \  \cite{zhang}, in dimension two, with a viscosity coefficient depending on the fluid density and the initial velocity in $H^1(\mathbb{R}^2)$; \ \cite{glz}, for spherically symmetric fluids, in dimension two and three, with both viscosity coefficients depending on the density; \ \cite{teixeira}, in the half-space in dimension three with the Navier boundary condition and the initial velocity in $H^1$. 

For the non-baratropic fluids, i.e. when the pressure depends also on the energy (in addition to the dependency on the density) extra difficulties appear to obtain the need estimates to show the lagrangian structure. It is necessary to conveniently extend some estimates presented in \cite{hoff-heatconducting}. We use techniques showed in \cite{hoff-santos} and \cite{zhouping} to get estimates in $L^2$ spaces for the material derivatives of speed and internal energy.


%


\smallskip

For convenience, let us state the mains properties we shall use here of the solution to the system \eqref{masaeq}-\eqref{energiaeq} obtained by \cite{hoff-heatconducting}: 


\begin{teor}\label{solucaofraca} \cite{hoff-heatconducting}.
Let $C_0$ be the quantitity defined in \eqref{E0} but without the norm $\|\vel_0\|_{H^1}$, and assume that  $\lambda<(1+\sqrt 2)\mu$. Let positive constants $\overline{\rho}>\tilde\rho>\underline\rho>0$ and $\tilde e>e_1>\underline{e}>0$ be given. Then there are positive constants $C$, $\epsilon$ such that if the initial data $(\rho_0,e_0,\vel_0)$ satisfies $C_0\le \epsilon$ and $\essinf e_0\ge e_1$, then the initial value problem \eqref{masaeq}-\eqref{energiaeq}, \eqref{ci}, has a global weak solution $(\rho,e,\vel)$ with the following properties:
\begin{equation}\label{regularidade}
\begin{array}{c}
\rho-\tilde\rho\in C([0,\infty);H^{-1}(\R^2)), \quad 
\rho(\cdot,t)-\tilde\rho \in (L^2\cap L^\infty)(\R^2))\ t\ge0,\\
\vel, \ e-\tilde e\in C((0,\infty);L^2(\R^2)),\\ 
\rho\in [\underline\rho,\overline\rho] \ a.e., \quad e(\cdot,t)\ge\underline{e} \ a.e.,\\
\vel(\cdot,t), \ F(\cdot,t), \ \omega(\cdot,t), \ e(\cdot,t)-\tilde e\in
H^1(\mathbb{R}^2), \quad t>0
\end{array}
\end{equation}
where $F$ (as already said in the Introduction) denotes the so called {\rm effective viscous flux}, i.e. the quantity $F := (\mu + \lambda)\mbox{div}\,\vel - (P-\tilde P)$, being $\tilde P:=P(\tilde\rho,\tilde e)$, and $\omega$ denotes the vorticity matrix (i.e. $\omega=(\omega^{i,j}), \ \omega^{i,j}=u^i_{x_j}-u^j_{x_i}$);
\begin{equation}\label{energia}
\begin{array}{l}
\sup\limits_{0<\tau<t}\int \left[(\rho-\tilde\rho)^2+|\vel
|^2+(e-\tilde e)^2\right](\x,\tau)W d\x\\
\quad +\int_0^t\int \left[|\nabla\vel|^2+|\nabla e|^2\right](\x,\tau)Wd\x
d\tau \le CC_0
\end{array}
\end{equation}
for any $t>0$, where $W\equiv W(x,\tau):=(1+|\x|^2)^l$ if $\tau\le 1$ and W:=1 elsewhere;
\begin{equation}\label{conv1velsinsob}
\sup\limits_{0<t<1} t\int|\nabla \vel|^2d\x+\int_0^1\int t|\dot\vel|^2d\x dt\le CC_0;
\end{equation}
\begin{equation}\label{conv2velsinsob}
\sup\limits_{0<t<1}t^{2}\int|\dot \vel|^2d\x+\int_0^1\int t^{2}
|\nabla\dot\vel|^2d\x dt\le CC_0.
\end{equation}
In addition, for global positive constants $\theta, q$, 
\begin{equation}\label{einfinito}
\|e(\cdot,t)-\tilde e\|_{L^\infty({\mathbb{R}^2})}\le CC_0^\theta t^{-q}, \quad 0<t\le1;
\end{equation}
\begin{equation}\label{holdercontinuidade}
\left\langle\,\vel\, \,\right\rangle^{\alpha,\alpha/(2+2\alpha)}_{\R^2\times[t,\infty)}, \  \left\langle\, e\,\right\rangle^{\alpha,\alpha/(2+2\alpha)}_{\R^2\times[t,\infty)}
\le CC_0^\theta t^{-q}, \quad 0<t\le1,
\end{equation}
where $C$ may depend additionally on $t$ and $\alpha$, and $\left\langle\cdot\right\rangle^{\alpha,\beta}$ denotes the H\"older semi-norm with exponent $\alpha$ in the $x$ variable and exponent $\beta$ in the $t$ variable;

Furthermore, the solution $(\rho,\vel,e)$ is the limit of smooth approximate solutions $(\rho^\delta,\vel^\delta,e^\delta)$, $\delta\to0$, which satisfy the estimates \eqref{energia}-\eqref{holdercontinuidade} with the constants $C,\theta,q$ on the right hand side of these estimates independent of $\delta$.
\end{teor}

Assuming the initial velocity $\vel_0$ in $H^1$, the estimates \eqref{conv1velsinsob}, \eqref{conv2velsinsob} can be improved such that the powers in $t$ decrease by one, and we have also a similar estimate to \eqref{conv1velsinsob} for the internal energy $e$, i.e. we have the following result which we prove in Section \ref{evidence}:  

\begin{teor}\label{teo3} Under the same hypothesis and notations in Theorem \ref{solucaofraca}, if the initial velocity $\vel_0$ is in the Sobolev space $H^1(\mathbb{R}^2)$ and $C_0\le\epsilon$ (possibly with a smaller $\epsilon$ than that in Theorem \ref{solucaofraca}, and with $C_0$, defined in \eqref{E0}, including now the norm $\|\vel_0\|_{H^1(\mathbb{R}^2)}$), then we have the following estimates on the approximated solutions $(\rho^\delta,\vel^\delta,e^\delta)$ stated in Theorem \ref{solucaofraca}, with the constant $C$ as above (in particular, independent of $\delta$):
\begin{equation}\label{conv1vel}
\sup\limits_{0<t<1}\int|\nabla \vel^\delta|^2d\x+\int_0^1\int
 |\dot\vel^\delta|^2d\x dt\le CC_0^\theta;
\end{equation}
\begin{equation}\label{conv2vel}
\sup\limits_{0<t<1}t\int|\dot\vel^\delta|^2d\x+\int_0^1\int t |\nabla\dot\vel^\delta|^2d\x
dt\le CC_0^\theta;
\end{equation}
\begin{equation}\label{conv1ener}
\sup\limits_{0<t<1}\int|\nabla e^\delta|^2d\x+\int_0^1\int |\dot e^\delta|^2d\x dt\le CC_0^\theta.
\end{equation}
\end{teor}


The estimates \eqref{conv1vel} and \eqref{conv2vel} imply the lagrangian structure (uniqueness of particle paths) in initial time $t=0$, as we show in Section \ref{ls}, following \cite{hoff-santos}. 
 The estimates \eqref{conv1vel}, \eqref{conv2vel} were obtained in \cite{hoff-2dim} in the case that the pressure is a function of the density only (more precisely, of the form $P(\rho) = A \rho ^\gamma$, for constants $ A>0$ and $\gamma>1$). Here, since the pressure depends also on the energy, we have extra difficulties to obtain them. For instance, in our arguments (see Section \ref{evidence}) we need to use \eqref{conv1ener} to obtain \eqref{conv1vel}, \eqref{conv2vel}, i.e. due to the pressure term the three estimates \eqref{conv1vel}-\eqref{conv1ener} are entailed to each other.

\smallskip


The remainder of this paper is organized as follows: in Section \ref{preliminaries} we collect some facts we shall use in the next Sections \ref{evidence}, \ref{ls}. In Section \ref{evidence} we prove the  estimates \eqref{conv1vel}-\eqref{conv1ener} and in Section \ref{ls} we prove Theorem \ref{lagrangiana}.

\section{Preliminaries}
\label{preliminaries}

Throughout this paper we shall use some classical estimates which we recall for the convenience of the reader. 

\smallskip

The Morrey's inequality for a function $f$ in the Sobolev space $W^{1,p}(\mathbb{R}^2)$ with $p>2$ is
\begin{equation}\label{morrey}
\langle f\rangle^\alpha\le C\|\nabla f\|_{L^p(\mathbb{R}^2)}
\end{equation}
where $\alpha=1-\frac{2}{p}$, $\langle \cdot\rangle^\alpha$ denotes the H\"older semi-norm and $C$ is a constant depending only on $p$. As a consequence, we have the estimate 
\begin{equation}\label{infty2p}
\inorma{f}\le C\left(\xnorma{\nabla f}{p}+\dosnorma{f}\right),
\end{equation}
which can be obtained from \eqref{morrey} by writing $f(x)=-\!\!\!\!\!\int_{B_1(0)}|x-y|^\alpha((f(x)-f(y)/|x-y|^\alpha))dy+-\!\!\!\!\!\int_{B_1(0)}|f(y)|dy$ and properly estimating these integrals.
 
We shall use several times the interpolation inequality
\begin{equation}\label{g-n}
\xnorma{f}{p}^p\le C\dosnorma{f}^{2}\dosnorma{\nabla f}^{p-2}.
\end{equation}

\smallskip

A very useful expedient introduced by Hoff (see e.g. \cite{hoff-95}) is to write the momentum equation (see \eqref{momentoeqc}) as 
\begin{equation}
\label{me Fw}
\rho\dot\vel = \nabla F + \mu\mbox{div}\omega
\end{equation} 
i.e. $\rho\dot u^j = F_{x_j} + \mu\omega^{j,k}_{x_k}$, $j=1,2$, where $F$ is the {\em effective viscous flux}, i.e. the quantity $F := (\mu + \lambda)\mbox{div}\,\vel - (P-\tilde P)$ (mentioned earlier), $\tilde P:=P(\tilde\rho,\tilde e)$, and $\omega\equiv (\omega^{j,k})$, $j,k=1,2$, is the vorticity matrix, i.e. $\omega^{j,k}=u^j_{x_j}-u^k_{x_j}$. \ 
Indeed, applying the $div$ and the $curl$ operators to \eqref{me Fw} we obtain
\begin{equation}
\Delta F=\mbox{div}(\rho\dot\vel), \quad\quad \mu\Delta\omega=\mbox{curl}(\rho\dot\vel) 
\end{equation}
where the last equation means $\mu\Delta\omega^{j,k}=\rho\dot u^j_{x_k}-\rho\dot u^k_{x_j}$, $j,k=1,2$. \ Then by elliptic theory, given any $p\in (1,\infty)$, there is a constant $C$ such that
\begin{equation}
\label{F w Lp}
\xnorma{\nabla F(\cdot,t)}{p}, \ \xnorma{\omega^{j,k}(\cdot,t)}{p} \ \le\ C\xnorma{\rho\dot{\vel}(\cdot,t)}{p}, \quad t>0.
\end{equation}
On the other hand, from the identity $\Delta u^j=(\lambda+\mu)^{-1}F_{x_j}+\omega^{j,k}_{x_k}+(\lambda+\mu)^{-1}(P-\tilde P)_{x_j}$ it follows that
\begin{equation}
\label{nabla u Lp}
\begin{array}{rl}
\pnorma{\nabla\vel(\cdot,t)} \le& C(\pnorma{F(\cdot,t)}+\pnorma{\omega(\cdot,t)}\\
& +\pnorma{(P-\tilde P)(\cdot,t)}), \quad\quad\quad t>0.
\end{array}
\end{equation}
Furthermore, writing $(\lambda+\mu)^{-1}F_{x_j}+\omega^{j,k}_{x_k}=\Delta u^j_{F,\omega}$ and $(\lambda+\mu)^{-1}(P-\tilde P)_{x_j}=\Delta u^j_{P}$, we have $\vel=\vel_{F,\omega}+\vel_P$, with $\vel_{F,\omega}$ satisfying the estimate 
\begin{equation}
\label{d2 ufw}
\xnorma{D^2 \vel_{F,\omega}(\cdot,t)}{p} \le C\xnorma{\rho\dot{\vel}(\cdot,t)}{p}, \quad t>0,
\end{equation}
in virtue of \eqref{F w Lp}. The inequalities \eqref{F w Lp}-\eqref{d2 ufw} will be used in the next sections. Regarding the part $\vel_{P}$ we have the following (leading to \eqref{uP LL} below):

\smallskip

Let us denote by $LL$ the the space of log-lipschitzian functions in $\mathbb{R}^2$,
i.e., the space of functions (or vector functions) $f$ defined in $\mathbb{R}^2$ such that the norm
\[
\llnorma{f}:=\llseminorma{f}+\inorma{f}
\]
is finite, where $\llseminorma{\cdot}$ denotes {\em log-Lipschitzian seminorm} defined by
\[
\llseminorma{f}:=\sup\limits_{0<|\x-\y|\le1}\dfrac{|f(\x)-f(\y)|}{m(|\x-\y|)}
\]
being
\[
m(r):=\begin{cases}
r(1-\log r),&\text{ if }0<r\le1\\
r,&\text{ if }r>1.
\end{cases}
\]
Then we have the following result:
 
\begin{lema}\label{solfundamental} 
Let $\Gamma$ denote the fundamental solution of the laplacian in $\mathbb{R}^2$.

\noindent
1. \ If $1\le p_1<2<p_2\le\infty$ and $f\in L^{p_1}(\mathbb{R}^2)\cap L^{p_2}(\mathbb{R}^2)$ then the vector field $f*\nabla\Gamma$, where $*$ denotes the standard convolution in $\mathbb{R}^2$ (i.e. $(f*\nabla\Gamma)(x)=\int_{\mathbb{R}^2}f(x-y)\nabla\Gamma(y)dy$, $x\in\mathbb{R}^2$) belongs to $L^\infty(\mathbb{R}^2)$ and
\begin{equation}
\label{lipschitz}
\inorma{f*\nabla\Gamma}\le C(\xnorma{f}{p_1}+\xnorma{f}{p_2}),
\end{equation}
where $C$ is a constant depending only on $p_1$ and $p_2$.

\noindent
2. \ If $1\le p<2$ e $f\in L^{p}(\mathbb{R}^2)\cap L^{\infty}(\mathbb{R}^2)$ then $f*\nabla\Gamma\in LL$ and
\begin{equation}
\label{loglip}
\llnorma{f*\nabla\Gamma}\le C(\xnorma{f}{p}+\inorma{f}),
\end{equation}
where $C$ is a constant depending only on $p$.
\end{lema}

\begin{obs} This lemma holds true in $\mathbb{R}^n$, with the same proof, replacing the conditions on the $p's$ by $p_1<n<p_2$ and $p<n$.
\end{obs}

\begin{proof} To the first estimate, separate the integral $f*\nabla\Gamma$ in the ball $B_1(0)$ and in its complementary. Then, estimate each integral by applying the H\"older's inequality in a convenient way. As for the second estimate, take $x, y\in \mathbb{R}^n$ with $\varepsilon = |x-y|\leq 1$, $\bar x=\dfrac{x-y}{2}$, and separate the integral in $(f*\Gamma_{x_j})(x)-(f*\Gamma_{x_j})(y)$, $j=1,2,...,n$, in the balls $B_\varepsilon(\bar x)$, $B_2(\bar x)\setminus B_\varepsilon(\bar x)$ and in $B_2(\bar x)^c$. Then it is possible to estimate the integral in these sets, respectively, by $\varepsilon\|f\|_{L^{\infty}(\mathbb{R}^n)}$, $\varepsilon(\ln 3-\ln\varepsilon)\|f\|_{L^\infty(\mathbb{R}^n)}$ and $\varepsilon \|f\|_{L^{p}(\mathbb{R}^n)}$, times some constant. 
\end{proof}

As a corollary of Lema \ref{solfundamental}, given any $p\in [1,2)$, we obtain the following estimate for the second part $\vel_P$ in the decomposition $\vel=\vel_{F+\omega}+\vel_P$ introduced above:
\begin{equation}
\label{uP LL}
\llnorma{\vel_P(\cdot,t)}\le C(\xnorma{(P-\tilde P)(\cdot,t)}{p}+\inorma{(P-\tilde P)(\cdot,t)}), \quad t>0.
\end{equation}
Indeed, in the above decomposition we can take $\vel_P$ as $\vel_P=\Gamma*\nabla(P-\tilde P)=(\nabla\Gamma)*(P-\tilde P)=(P-\tilde P)*\nabla\Gamma$. Thus, \eqref{uP LL} is a consequence of  
\eqref{loglip}, since by \eqref{energia} we have $(P-\tilde P)(\cdot,t)\in L^p(\mathbb{R}^2)$, for any $p\in [2/(1+l),2]$, and $(P-\tilde)(\cdot,t)\in L^\infty(\mathbb{R}^2)$ by \eqref{einfinito}. In fact, we shall need $\llnorma{\vel_P(\cdot,t)}$ to be locally integrable in time. The estimate \eqref{energia} gives that $\xnorma{(P-\tilde P)(\cdot,t)}{p}$ has this property, so by \eqref{uP LL}, to have that it is enough that $\inorma{(P-\tilde P)(\cdot,t)})$ to be locally integrable. Fortunately, we have
\begin{lema} \cite[Lemma 4.4]{hoff-heatconducting} \ 
\begin{equation}\label{integrale}
\int_0^1\inorma{e-\tilde e(\cdot,t)}dt\le CC_0^\theta.
\end{equation}
\end{lema}
\noindent
Then, combining \eqref{integrale} with $\underline\rho\le\rho\le\overline\rho$ (see \eqref{regularidade}), it follows that $\inorma{(P-\tilde P)(\cdot,t)}$ is locally integrable.

\bigskip

\section{Estimates of convective terms}\label{evidence}

In this Section we show the estimates \eqref{conv1vel}-\eqref{conv1ener}. For convenience we omit the superscript $\delta$ in the approximate solution $(\rho^\delta,\vel^\delta,e^\delta)$.

We begin by defining the functionals
\[
\begin{array}{rcl}
B_0(t)&=&\supintnormard{t}{\nabla \vel}{\dot\vel},\\
B_1(t)&=&\supintnormard{t}{\nabla e}{\dot e},\\
B(t)&=&B_0(t)+B_1(t).
\end{array}
\]

\begin{lema}\label{b0}
If $t\le 1$, then
\begin{align*}
B_0(t) \le C(C_0^{\theta}+\sum\limits_{k>1}B(t)^{k})
\end{align*}
\end{lema}
where $\sum\limits_{k>1}$ is a finite sum over real indexes $k>1$.

\begin{proof}
Multiplying \eqref{momentoeqc} by $\dot u^j$, we have
\begin{align*}
\rho(\dot u^j)^2&=-P_{x_j}\dot u^j+\left(\mu\triangle
u^j+\lambda(\diver \vel)_{x_j} \right){u^j}_t\\&\qquad+\left(\mu\triangle
u^j+\lambda(\diver \vel)_{x_j} \right)\nabla u^j\cdot \vel
\end{align*}
and integrating by parts and summing on $j=1,2$, we obtain
\begin{equation}\label{lemamom}
\begin{array}{rl}
\int^t_0\int \rho|\dot{\vel}|^2d\x d\tau &+\frac
12\int(\mu|\nabla\vel|^2+\lambda(\diver \vel)^2)d\x|_0^t\\ &= \qquad +\int_0^t\int (P-\tilde
P)\diver \dot\vel d\x \\
&\quad+ \int_0^t\int \left(\mu\triangle u^j+\lambda(\diver
\vel)_{x_j} \right)\nabla u^j\cdot \vel d\x d\tau.
\end{array}
\end{equation}

%

The term above with the pressure can be written as
\begin{equation}\label{preswpont}
\begin{array}{rl}
\int_0^t\int(P-\tilde P)\diver \dot\vel d\x&=\int_0^t\int(P-\tilde P)\partial_t(\diver \vel)d\x d\tau\\&\qquad\qquad+\int_0^t\int(P-\tilde P)\diver \left((\nabla \vel)\vel\right) d\x d\tau\\
&=\left(\int(P-\tilde P)\diver \vel d\x\right)|_0^t\\
&\quad\quad+\int_0^t\int (P-\tilde P)\diver\left(\vel\diver \vel-(\nabla \vel)\vel\right)d\x d\tau\\
&\quad\quad+\int_0^t\int (P_\rho\rho(\diver\vel)^2- P_e\dot e\diver \vel)d\x d\tau\\
&=\int (P-\tilde P)\diver \vel d\x\\
&\quad\quad+\int_0^t\int  (P-\tilde P)\diver\left((\nabla \vel)\vel-\vel\diver \vel\right)d\x d\tau\\
&\quad\quad+\int_0^t\int (P_\rho\rho(\diver\vel)^2- P_e\dot e\diver \vel)d\x d\tau\\
\end{array}
\end{equation}
and, by the identity
\[
\diver\left((\nabla \vel)\vel-\vel\diver \vel\right)=(\diver \vel)^2-u_{x_j}^k{u}_{x_k}^j,
\]
we have that the modulus of the second term on the last expression above is bounded by
\[
\int_0^t\int|P-\tilde P||\nabla \vel|^2d\x.
\]
In addition, the fourth term in \eqref{preswpont} can be written as
\begin{align*}
\int_0^t\int P_\rho\rho(\diver\vel)^2 d\x
d\tau&=(\gamma-1)\int_0^t\int \rho e(\diver\vel)^2 d\x d\tau\\
&=\int_0^t\int P(\diver\vel)^2 d\x d\tau\\
&=\int_0^t\int (P-\tilde P)(\diver\vel)^2 d\x
d\tau\\
&\quad\quad+\int_0^t\int \tilde P(\diver\vel)^2 d\x d\tau.
\end{align*}

Then, since that the last integral in \eqref{lemamom} is easily bounded by $C\int\int |\nabla\vel|^3d\x d\tau$, we have that \eqref{lemamom} can be bounded as
\begin{equation}\label{momentopresao}
\begin{array}{rl}
\int^t_0\int \rho|\vel|^2d\x d\tau &+\frac
12\int \mu|\nabla\vel|^2+\lambda(\diver \vel)^2d\x \le C\dosnorma{\nabla\vel_0}+\int (P-\tilde
P)\diver \vel d\x \\
&\quad+ C\int_0^t\int  |\nabla \vel|^2d\x
d\tau+C\int_0^t\int  |\nabla \vel|^3d\x d\tau\\
&\quad +C\int_0^t\int  |P-\tilde P||\nabla \vel|^2d\x
d\tau-(\gamma-1)\int_0^t\int \rho\dot e\diver\vel d\x d\tau
\end{array}
\end{equation}
Next, let us estimate each integral on the right hand side in this inequality
The two first integrals in \eqref{momentopresao} are easily bounded using the energy estimates \eqref{energia} for $\vel$. The third integral is estimated by
\begin{align*}
C\int_0^t\int |\nabla \vel|^3d\x d\tau &\le
C\left(\int_0^t\int |\nabla \vel|^4d\x
d\tau\right)^{1/2}\left(\int_0^t\int|\nabla
\vel|^2d\x d\tau\right)^{1/2}\\
&\le CC_0^{1/2}\left(\int_0^t\int |\nabla \vel|^4d\x d\tau\right)^{1/2}\\
&=CC_0^{1/2}\left(\int_0^t \xnorma{\nabla\vel}{4}^4d\tau\right)^{1/2}\\
\end{align*}

Using that $\rho$ is bounded from above and below (see \eqref{regularidade}) and \eqref{energia}, we have
\begin{align*}
\int_0^t\int |P-\tilde P||\nabla \vel|^2d\x d\tau&\le
C\int_0^t\int |\nabla \vel|^2d\x d\tau\\
& \quad+C\int_0^t\int |e-\tilde e||\nabla \vel|^2d\x
d\tau\\
&\le CC_0+C\int_0^t \dosnorma{e-\tilde e}\xnorma{\nabla
\vel}{4}^2 d\tau\\
&\le CC_0+CC_0\int_0^t \xnorma{\nabla
\vel}{4}^2 d\tau\\
&\le CC_0+CC_0t^{1/2}\left(\int_0^t \xnorma{\nabla
\vel}{4}^4d\tau\right)^{1/2} 
\end{align*}


The terms in the two first integrals on the right hand side of \eqref{preswpont}are easily bounded using the energy estimates or can be combined with the left hand side. The last integral is bounded by $ CC_0^{1/2}B_1(t)^{1/2}$, by H\"older's inequality and the definition of $B_1$. Therefore, we just need to estimate the term of the $L^4$ norm:
\begin{align*}
&\int_0^t\int |\nabla \vel|^4d\x d\tau \le \int_0^t\int (F^4+|\omega|^4+|P-\tilde P|^4)d\x d\tau\\
\le & C\int_0^t\left(\int F^2d\x\right)\left(\int |\nabla F|^2d\x\right)+\left(\int |\omega|^2d\x\right)\left(\int |\nabla \omega|^2d\x\right)d\tau\\
&\qquad +(CC_0t+CC_0^2)\\
\le& C\int_0^t\left(\dosnorma{\nabla\vel}^2+\dosnorma{P-\tilde
P}^2\right)\dosnorma{\rho^{1/2}\dot \vel}^2d\tau\\
&\qquad +(CC_0t+CC_0^2)\\
\le& C(B_0(t)^2+C_0B_0+C_0+C_0^2).
\end{align*}

Finally, we have
\begin{align*}
B_0(t) \le  C\dosnorma{\nabla\vel_0}^2+ &CC_0+CC_0^{1/2}(B_0(t)^2+C_0B_0+C_0+C_0^2)^{1/2}\\
& +C(B_0(t)^2+C_0B_0+C_0+C_0^2)+CC_0^{1/2}B_1(t)^{1/2}\\
\qquad \le C(C_0^{\theta}+\sum\limits_{k>1}B(t)^{k})
\end{align*}
where the sum over $k$ is a finite sum and $\theta>0$ is a convenient constant.
\end{proof}

\begin{lema}
\[
B_1(t)\le
C\left[C_0^\theta+\sum\limits_{k>1}B(t)^k\right],
\]
where $\sum\limits_{k>1}$ is a finite sum over real indexes $k>1$.
\end{lema}
\begin{proof}
Analogously to the proof of Lemma \ref{b0}, from \eqref{energiaeq} we have
\begin{align*}
B_1(t)&\le C\dosnorma{\nabla e_0}^2+C\int_0^t\int |\nabla
e|^2|\nabla \vel|d\x
d\tau\\
&\quad +C\int_0^t\int (e-\tilde e)^4+|\nabla
\vel|^4d\x d\tau\\
\end{align*}
As mentioned above,
\begin{align*}
\int_0^t\xnorma{\nabla\vel}{4}^4d\tau\le B_0(t)^2+CC_0B_0(t)+CC_0+CC_0^2.
\end{align*}
In addition,
\begin{align*}
\int_0^t\int (e-\tilde e)^4 d\x d\tau&\le \int_0^t\dosnorma{e-\tilde
e}^2\dosnorma{\nabla e}^2d\tau\\
\le CC_0.
\end{align*}
Then, it remains only to estimate the term $\int\int|\nabla e|^2|\nabla \vel|d\x d\tau$. Notice that
 \begin{align*}
 \int_0^t\int |\nabla e|^2|\nabla \vel| d\x d\tau&\le C\int_0^t\int
 |\nabla e|^{8/3}+|\nabla\vel|^4d\x d\tau
 \end{align*}
and, since,
\begin{align*}
&\int_0^t\int |\nabla e|^{8/3}d\x d\tau\le C\int_0^t
\dosnorma{\nabla e}^2 \dosnorma{D^2 e}^{2/3}\\
&\le \sup\limits_{0\le \tau \le t}\dosnorma{\nabla
e}^{2/3}\left(\int_0^t \dosnorma{\nabla
e}^2\right)^{2/3}\left(\int_0^t\int |D^2e|^2\right)^{1/3}\\
&\le CC_0^{2/3}B_1(t)^{2/3}\left(\int_0^t\int |D^2e|^2\right)^{1/3}\\
&\le CC_0^{2/3}B_1(t)^{2/3}\left(\int_0^t\int |\dot e|^2+|e|^2|\nabla\vel|^2+|\nabla \vel|^4\right)^{1/3}\\
&\le CC_0B_1(t)^{2/3}+CC_0^{2/3}B_1(t)^{2/3}\left(\int_0^t\int |\dot e|^2+|e-\tilde e|^2|\nabla\vel|^2+|\nabla \vel|^4\right)^{1/3}\\
&\le CC_0B_1(t)^{2/3}+CC_0^{2/3}B_1(t)^{2/3}\left(B_1(t)+B_0(t)^2+CC_0B_0(t)+CC_0+CC_0^2\right)^{1/3},
\end{align*}
choosing $\theta>0$ conveniently and applying the Young inequality, we can write
\[
B_1(t)\le
C\left[C_0^\theta+\sum\limits_{k>1}B(t)^k\right]
\]
\end{proof}

Combining the previous lemmas, we obtain the estimate 
\begin{equation}\label{stage1}
B(t)\le CC_0^{\theta},
\end{equation}
for any $0\le t\le1$.

We observe that by the above calculations,  we have also that the 4-norm $\int\int |\nabla\vel|^4d\x d\tau$ is bounded by $CC_0^\theta$. 

\bigskip

Next, we define a new function
\[
\begin{array}{rcl}
B_2(t)&=&\supintnormardsobconv{t}{\dot \vel}{\nabla\dot \vel}{},\\
\end{array}
\]
and show
\begin{lema}\label{conv2vellema}
\begin{align*}
B_2(t)\le CC_0^\theta
\end{align*}
\end{lema}
\begin{proof}
We begin by applying the operator $\tau\dot u^j(\partial_t+\diver(\cdot
\vel))$ to the moment equation \eqref{momentoeqc} to get
\begin{align*}
\frac{1}{2}\rho\frac{\partial}{\partial \tau}\left(\tau(\dot
u^j)^2\right)+\frac{1}{2}\tau\rho\vel\cdot\nabla (\dot u^j)^2&=\frac{1}{2}\rho(\dot
u^j)^2-\tau\dot u^j\left(P_{tx_j}+\diver{P_{x_j}\vel}\right)\\
&\quad+\mu\tau\dot u^j\left(\triangle u^j_t+\diver{\triangle
u^j\vel}\right)\\
&\quad+ \lambda\tau\dot
u^j\left((\diver\vel)_{tx_j}+\diver{(\diver\vel)_{x_j} \vel}\right).
\end{align*}
Calculating the integral in the variable $x$ and $t$ and using \eqref{masaeqc}, we have
\begin{align*}
\frac{1}{2}\int\rho\tau|\dot
u^j|^2&=\frac{1}{2}\int_0^t\int\rho|\dot
u^j|^2d\x d\tau-\int_0^t\int\tau\dot u^j\left(P_{tx_j}+\diver{P_{x_j}\vel}\right)d\x d\tau\\
&\quad+\mu\int_0^t\int\tau\dot u^j\left(\triangle
u^j_t+\diver{\triangle
u^j\vel}\right)d\x d\tau\\
&\quad+ \lambda\int_0^t\int\tau\dot
u^j\left((\diver\vel)_{tx_j}+\diver{(\diver\vel)_{x_j}
\vel}\right)d\x d\tau
\end{align*}
Then, summing in $j$ we get four integrals, which we estimate as follows. The first integral is bounded $CC_0^\theta$ by \eqref{stage1}. As for the second, using
\eqref{masaeq}, we have
\begin{align*}
&\int_0^t\int\tau\dot u^j\left(P_{tx_j}+\diver{P_{x_j}\vel}\right)d\x
d\tau\\
=&-\int_0^t\int\tau\diver \dot \vel P_td\x
d\tau-\int_0^t\int\tau P_{x_j}\dot u^j_{x_k}u^kd\x d\tau\\
=&-\int_0^t\int\tau\diver \dot \vel (P_\rho\rho_t+P_ee_t)d\x d\tau -\int_0^t\int\tau P(\dot u^j_{x_kx_j}u^k+\dot u^j_{x_k}u^k_{x_j})d\x d\tau\\
=&\int_0^t\int\tau P\diver \dot \vel\diver\vel d\x d\tau-\int_0^t\int\tau\diver \dot \vel(\vel\cdot\nabla P)d\x d\tau\\
&\quad -(\gamma-1)\int_0^t\int\tau \rho\dot e\diver \dot\vel d\x d\tau-\int_0^t\int\tau P(\dot u^j_{x_kx_j}u^k+\dot u^j_{x_k}u^k_{x_j})d\x d\tau\\
=&2\int_0^t\int\tau P\diver \dot \vel\diver\vel d\x d\tau-\int_0^t\int\tau P \dot u^j_{x_k}u^k_{x_j} d\x d\tau\\
&\quad -(\gamma-1)\int_0^t\int\tau\rho\diver\dot\vel\dot e d\x d\tau
\end{align*}

In the case of the viscosity terms, we can get
\begin{align*}
\int_0^t\int \tau\dot u^j&\left(\triangle u^j_t+\diver
(\triangle u^j\vel)\right)d\x d\tau=-\int_0^t\int
\tau|\nabla\dot\vel|^2d\x d\tau\\
&\quad +O\left(\int\int\tau|\nabla\dot\vel||\nabla\vel|^2d\x
d\tau\right).
\end{align*}
Similary, 
\begin{align*}
\int_0^t\int \tau\dot u^j&\left((\diver \vel)_{tx_j}+\diver
((\diver\vel)_{x_j}\vel)\right)d\x d\tau=-\int_0^t\int
\tau(\diver\dot\vel)^2d\x d\tau\\
&\quad +O\left(\int\int\tau|\nabla\dot\vel||\nabla\vel|^2d\x.
d\tau\right)
\end{align*}
With all this, we manage to conclude that
\begin{align*}
\tau&\int\rho|\dot \vel|^2d\x+\int_0^1\int \tau
|\nabla\dot\vel|^2d\x
d\tau\le CC_0^\theta\\
&\quad +C\int_0^t\int\tau|P-\tilde P||\nabla
\dot\vel|\nabla\vel|d\x d\tau+C\int_0^t\int\tau|\nabla
\dot\vel|\nabla\vel|^2d\x d\tau\\
&\quad +C\int_0^t\int\tau|\nabla\dot{\vel}||\dot e|d\x d\tau
\end{align*}
Terms with $\nabla\dot\vel$ can be absorved on the left side, the $L^2$ norm of $\dot e$ for the last integral above was estimated in \eqref{stage1} and the others are estimated as follows: first of all we have that $P-\tilde P=(\gamma-1)(\rho(e-\tilde e)+\tilde e(\rho-\tilde{\rho}))$, so that, by interpolation inequality and the energy estimates,
\begin{align*}
\xnorma{P-\tilde P}{4}&\le C(\xnorma{\rho-\tilde \rho}{4}+\xnorma{e-\tilde e}{4})\\
&\le C(\xnorma{\rho-\tilde \rho}{\infty}\dosnorma{\rho-\tilde \rho}+\dosnorma{e-\tilde e}^2\dosnorma{\nabla e}^2).
\end{align*}
So, we get
\begin{align*}
\int_0^t\xnorma{P-\tilde P}{4}^4\le CC_0^\theta.
\end{align*}
Using this, along with the observation before the definition of $B_2$, the energy estimates form \cite{hoff-heatconducting} and \eqref{stage1}, we obtain the desired result, i.e. 
\begin{align*}
B_2(t)\le CC_0^\theta.
\end{align*}
\end{proof}

\section{Lagrangian structure}
\label{ls}


In this section we show Theorem \ref{lagrangiana}.
We recall that all the previous estimates were obtained uniformly with respect to the approximate solutions. 


We shall perform the decomposition $\veld=\vel_P^\delta+\veld_{F,\omega}$ for the approximate solution $(\veld,\rho^\delta,e^\delta)$ and shall write $\vel_P^\delta=\vel_{P^\delta}^\delta$, for simplicity.

\begin{proof}[Proof of Theorem \ref{lagrangiana}]
The integral curve for the approximate field $\vel^\delta$ starting at $\x_0\in\mathbb{R}^2$ is given by
\begin{equation}\label{trajetoriadelta}
X^\delta(\x_0,t)=\x_0+\int_0^t\vel^\delta(X^\delta(\x_0,\tau),\tau)d\tau,\quad
t\ge0.
\end{equation}
The map $t\mapsto X^\delta(t,\x_0)$ is H\"older continuous, uniformly with respect to $\delta$. Indeed, for any $0\le t_1<t_2$,  by \eqref{infty2p} we have
\begin{align*}
|X^\delta(\x_0,t_1)-X^\delta(\x_0,t_2)|&\le
\int_{t_2}^{t_1}\inorma{\vel^\delta(\cdot,t)}d\tau\\
&\le
C\int_{t_2}^{t_1}\dosnorma{\vel^\delta(\cdot,\tau)}+\pnorma{\nabla\vel^\delta(\cdot,\tau)}d\tau
\end{align*}
and we can bound the $L^2$ norm above by $CC_0^\theta(t_2-t_1)^\gamma$, for some $\gamma\in (0,1)$(independent of $\delta$) exactly as in \cite[(3.9)]{hoff-heatconducting}. Using \eqref{nabla u Lp} and \eqref{infty2p}, the $L^p$ norm above can be bound as
\begin{align*}
\pnorma{\nabla\vel^\delta}&\le
\pnorma{F^\delta}+\pnorma{\omega^\delta}+\pnorma{P^\delta-\tilde P}\\
&\le CC_0^\theta \left(1+\int|\nabla\vel^\delta|^2d\x\right)^{(1-\eta)/2}\left(\int\rho^\delta|\dot\vel^\delta|^2d\x\right)^{\eta/2}\\
&\qquad+C\pnorma{\rho^\delta-\tilde\rho}+C\pnorma{e^\delta-\tilde e}\\
&\le CC_0^\theta \left(1+\int|\nabla\vel^\delta|^2d\x\right)^{(1-\eta)/2}\left(\int\rho^\delta|\dot\vel^\delta|^2d\x\right)^{\eta/2}\\
&\qquad+C\dosnorma{\rho^\delta-\tilde\rho}+C\dosnorma{e^\delta-\tilde e}^{1-\eta}\dosnorma{\nabla e^\delta}^\eta\\
&\le CC_0^\theta \left(1+\int|\nabla\vel^\delta|^2d\x\right)^{(1-\eta)/2}\left(\int\rho^\delta|\dot\vel^\delta|^2d\x\right)^{\eta/2}\\
&\qquad+CC_0^\theta(1+\dosnorma{\nabla e^\delta}^\eta).
\end{align*}
where $\eta=(p-2)/p$. Then
\begin{align*}
&\int_{t_1}^{t_2}\pnorma{\nabla\vel^\delta(\cdot,\tau)}d\tau\\ 
&\le CC_0^\theta\int_{t_1}^{t_2} \left(1+\int|\nabla\vel^\delta|^2d\x\right)^{(1-\eta)/2}\left(\int\rho^\delta|\dot\vel^\delta|^2d\x\right)^{\eta/2}d\tau\\
&\qquad+CC_0^\theta\int_{t_1}^{t_2}(1+\dosnorma{\nabla e^\delta}^\eta)\\
&\le CC_0^\theta[(t_2-t_1)+(t_2-t_1)^{\gamma_1}],
\end{align*}
for some constant $\gamma_1\in (0,1)$ (independent of $\delta$). The above bounds show that $X^\delta(\x_0,\cdot)$ is H\"older continuous with respect to $\delta$, which guarantees the existence of a sequence $X^{\delta_n}(\x_0,\cdot)$ converging (as $\delta_n\to0$) to a H\"older continuous map $X(\x_0,\cdot)$, uniformly in compacts in $[0,\infty)$. Then passing to the limit in \eqref{trajetoriadelta} we get \eqref{trajetoria}.
\end{proof}


To prove the uniqueness of the map $X(\cdot,\x_0)$ satisfying \eqref{trajetoria}, we restrict $t$ close to zero, since the uniqueness of a trajectory starting in $t>0$ is
considerably simpler. We omit some details facilitating the exposure, which can be seen in
\cite{hoff-santos}. Let $X_1(\y_1,\cdot)$ and $X_2(\y_2,\cdot)$ two integral curves starting at $\y_1$ and $\y_2$, respectively, when $t=0$, i.e. $X_1,X_2$ satisfy \eqref{trajetoria} with $X$ replaced by $X_1,X_2$ and $\x_0$ by $\y_1,\y_2$, respectively. Following \cite{hoff-santos}, we introduce the function
\[
g(t):=\dfrac{|\vel(X_2(t,\y_2),t)-\vel(X_1(t,\y_1),t)|}{m(|X_2(t,\y_2)-X_2(t,\y_2)|)}
\]
and prove first the following lemma:
\begin{lema}\label{gint} There is a constant $C$, independent of $\y_2\in \R^2$ such that
\begin{equation}
\label{g}
\int_0^1g(\tau)d\tau \le C.
\end{equation}
\end{lema}
\begin{proof}
By Fatou's lemma, it is enough to prove \eqref{g} with $\vel^\delta$ in place of $\vel$ and with
associated integral curves $X_1^\delta, X_2^\delta$ in place of $X_1,X_2$. Recalling the decomposition $\vel^\delta=\vel_{F,\omega}^\delta+\vel_P^\delta$, we observe that $g^\delta(t)\le
\llseminorma{\vel_P^\delta}+\llseminorma{\vel_{F,\omega}^\delta}$. By \eqref{uP LL}, \eqref{energia} and \eqref{integrale}, we have 
\begin{equation}\label{llintegrabilidad}
\int_0^1\llseminorma{\vel_P^\delta}dt 
\le C.
\end{equation}

As for $\llseminorma{\vel_{F,\omega}^\delta}$, we can estimate $\inorma{\nabla\vel^\delta_{F,\omega}}$ using \eqref{infty2p} and then \eqref{d2 ufw}. Thus, since the $LL$-semi norm is bounded by the Lipschitzian semi norm, we obtain
$$
\llseminorma{\vel_{F,\omega}^\delta} \le C(\dosnorma{\nabla\vel^\delta_{F,\omega}}+\pnorma{\dot{\vel}^\delta}).
$$
Now, 
$$
\dosnorma{\nabla\vel^\delta_{F,\omega}}\le C(\dosnorma{F^\delta}+\dosnorma{\omega^\delta})\le C(\dosnorma{\nabla\vel^\delta}+\dosnorma{P^\delta-\tilde P}),
$$
and $\dosnorma{P^\delta-\tilde P}\le CC_0^\theta$, then, using \eqref{g-n} it follows that
$$
\llseminorma{\vel_{F,\omega}^\delta} 
\le C(C_0^\theta+\dosnorma{\nabla\vel^\delta}+\dosnorma{\dot\vel^\delta}^{1-\eta}\dosnorma{\nabla\dot\vel^\delta}^{\eta})
$$
where $\eta=\frac{p-2}{p}$. At this point the estimates \eqref{conv1vel}, \eqref{conv2vel} are crucial to obtain that
\begin{align*}
&\int_0^{1}\dosnorma{\rho^\delta\dot\vel^\delta}^{1-\eta}\dosnorma{\nabla\dot\vel^\delta}^{\eta}dt\\
\le & C(\int_0^1t^{-\eta}dt)^{1/2}(\int_0^1\int\rho^\delta|\dot\vel^\delta|^2dt)^{1-\eta}
(\int_0^{t}\int t|\nabla\dot\vel^\delta|^2dt)^{\eta}<\infty,
\end{align*}
since we can choose $p>2$ sufficiently close to $2$. Therefore, this and the above estimates, together with the energy estimate, show that 
$$
\int_0^1\llseminorma{\vel_{F,\omega}^\delta}dt<\infty.
$$
\end{proof}

\medskip

Returning to the integral curves $X_1$ e $X_2$,  we have that
\[
|X_2(t,\y_2)-X_1(t,\y_1)|\le|\y_2-\y_1|+\int_0^t
g(\tau)m(|X_2(t,\y_2)-X_1(t,\y_1)|)d\tau
\]
so, by Osgood's Lemma (\cite{chemin}, \cite{flett}) and by Lemma \ref{gint}, we obtain that
\begin{equation}\label{osgood}
|X_2(t,\y_2)-X_1(t,\y_1)|\le\exp(1-e^{-\int_0^tgd\tau})|\y_2-\y_1|^{\exp(-\int_0^tgd\tau)},
\end{equation}
which, in particular, implies $X_1=X_2$ if $\y_1=\y_2$.

\medskip

Next, we show the second claim of Theorem \ref{lagrangiana}. First, we show that for fixed $t>0$, the map $X(t,\cdot):\x\mapsto X(t,\x)$ is injective. Suppose that $X(t,\y_1)=X(t,\y_2)$ for some $\y_1,\y_2\in\mathbb{R}^2$. For $t'\in(0,t)$, writing
\begin{align*}
X(t',\y_1)-X(t',\y_2)&=X(t',\y_1)-X(t,\y_1)+X(t,\y_2)-X(t',\y_2)\\
&=\int_{t'}^t \vel(X(\tau,\y_2),\tau)-\vel(X(\tau,\y_1),\tau)d\tau,
\end{align*}
we have, as above, that
\[
|X(t',\y_1)-X(t',\y_2)|\le \int_{t'}^t
g(\tau)m(X(\tau,\y_1)-X(\tau,\y_2))d\tau.
\]
Therefore, $X(t',\y_1)=X(t',\y_2)$ for all 
$0<t'\le t$. Then, by the continuity of the maps $X(\cdot,\y_1), X(\cdot,\y_2)$, it follows that $\y_1=X(0,\y_1)=X(0,\y_2)=\y_2$.

To prove that the map $\x\mapsto X(t,\x)$ is onto, we use the uniqueness of the particle paths. Indeed, for $\y\in\R^2$,  there is a curve $Y(s)=X(s;\y,t)$ with
$Y(t)=\y$, $s\in[0,t]$. The curves $Y(s)$ and $X(s,Y(0))$ satisfy the problem
\[
\begin{cases}
&\frac{d}{ds}Z(s)=\vel(Z(s),s)\\
&Z(0)=Y(0)
\end{cases}
\]
thus, $Y(s)=X(s,Y(0))$ for all $s\in [0,t]$. In particular, $\y=X(t,Y(0))$.

To conclude the proof of the claim 2 of Theorem \ref{lagrangiana}, it remains to show that the map $\x\mapsto X(t,\x)$ is open. Let  $A$ be an open set in $\R^2$. We fix
a $\z_1=X(t,\y_1)$ with $\y_1\in A$. We know that for any $\z\in\R^2$, there exists a curve $Y(s)=X(s;\x,t)$ defined for $s\in[0,t]$ and such that $Y(t)=\z$. In fact,
\begin{equation}\label{ycurva}
Y(s)=\z+\int_s^t\vel(X(\tau;\z,t),\tau)d\tau=\z+\int_s^t\vel(Y(\tau),\tau)d\tau
\end{equation}
As done before, we can see that
\[
\int_0^{t'}\inorma{\vel(\cdot,\tau)}d\tau\le CC_0^\theta
t'^{\gamma_1}
\]
and fixing $r>0$ such that $B_r(\y_1)\subset A$ we can chose
$t'\le [(r-r_1)/(2CC_0^\theta)]^{1/\gamma_1}$ with $0<r_1<r$ to obtain
\begin{equation}\label{uinfty}
\int_0^{t'}\inorma{\vel(\cdot,\tau)}d\tau\le \frac{r-r_1}{2}.
\end{equation}
On the other hand, since
\[
\frac{d}{d\tau}X(\tau,\y_1)=\vel(X(\tau,\y_1),\tau);\qquad
X(t,\y_1)=\z_1
\]
we have that
\begin{equation}\label{trajetoriat}
X(s,\y_1)=z_1-\int_s^t \vel(X(\tau,\y_1),\tau)d\tau
\end{equation}
and so,  substracting \eqref{ycurva} of \eqref{trajetoriat}
\begin{align*}
|X(s,\y_1)-Y(s)|&\le|\z_1-\z|+\int_s^t
|\vel(Y(\tau),\tau)-\vel(X(\tau,\y_1),\tau)|d\tau\\
&\le |\z_1-\z|+\int_s^t g(\tau)|m(Y(\tau)-X(\tau,\y_1))|d\tau
\end{align*}
and using again the Osgood's inequality
\[
|X(s,\y_1)-Y(s)|\le
\exp\left(1-e^{-\int_s^tg(\tau)d\tau}\right)|\z_1-\z|^{e^{\int_s^tg(\tau)d\tau}}
\]
for all $s\in[t',t)$. If we chose $a\le\left(
r_1\exp\left(1-e^{-\int_{t'}^tg(\tau)d\tau}\right)\right)^{1/\gamma_t}$
where $\gamma_t$ is a suitable constant  depending possibly
of $t$, we get
\[
|X(t',\y_1)-Y(t')|\le r_1.
\]
This, together with \eqref{uinfty} shows that
\[
|\y_1-Y(0)|\le r,
\]
which assures that $Y(0)\in B_r(\y_1)$ when $\z\in B_a(\z_1)$.
Finally, by the the uniqueness of the trajectories, we have that
$\z=Y(t)=X(t,Y(0))$, so $B_a(\z_1)\subset X(t,A)$, as
wanted.

\medskip

\noindent
{\em Proof of claim \ref{holdercompactos} of Theorem \ref{lagrangiana}}: \ The proof of \ref{holdercompactos} is also an application of the Osgood's inequality. Let $\y_1,\y_2\in \R^2$ and consider the
function
\[
\tilde
g(\tau)=\frac{|\vel(X(\tau,\y_2),\tau)-\vel(X(\tau,\y_1),\tau)|}{m(|X(\tau,\y_2)-X(\tau,\y_1)|)}
\]
As in Lemma  \ref{gint}, we can show that
\[
\int_0^t \tilde g(\tau)d\tau\le Ct^\gamma.
\]
On the other hand, for $t\in[t_1,t_2]$, we can write
\[
X(t,\y_2)-X(t,\y_1)=X(t_1,\y_2)-X(t_1,\y_1)+\int_{t_1}^{t}\vel(X(\tau,\y_2),\tau)-\vel(X(\tau,\y_1),\tau)d\tau.
\]
Thus, 
\begin{align*}
|X(t,\y_2)-X(t,\y_1)|&\le
|X(t_1,\y_2)-X(t_1,\y_1)|+\int_{t_1}^{t}\tilde
g(\tau)m(|X(\tau,\y_2)-X(\tau,\y_1)|)d\tau.
\end{align*}
Then
\begin{equation}\label{holdertrajetoria}
\begin{array}{rl}
|X(t,\y_2)-X(t,\y_1)|&\le \exp{\left(1-e^{-\int_{t_1}^t \tilde
g(\tau)d\tau}\right)}|X(t_1,\y_2)-X(t_1,\y_1)|^{e^{-\int_{t_1}^{t}\tilde
g(\tau)d\tau}}\\
&\le C(t)|X(t_1,\y_2)-X(t_1,\y_1)|^{e^{-Lt^\gamma}},
\end{array}
\end{equation}
for all $t\in [t_1,t_2]$. In particular, for $t=t_2$, as desired.

\medskip

\noindent
{\em Proof of claim \ref{curves H cont} of Theorem \ref{lagrangiana}}: \ Let the curve $\mathcal C$ be parametrized by a H\"older continuous $\varphi$ with exponent $\alpha$.
Defining $\psi_t(s)=X(t,\varphi(s))$ and using \eqref{holdertrajetoria}, we obtain
\[
|\psi_t(s_2)-\psi_t(s_1)|\le C(t)|\varphi(s_2)-\varphi(s_1)|^{e^{-L
t^\gamma}}\le C(t)|s_2-s_1|^{\alpha e^{-L t^\gamma}}.
\]
This proves the claim \ref{curves H cont}.



\bibliographystyle{plain}
\bibliography{bibfluidos}

\begin{thebibliography}{10}

\bibitem{anderson}
Jhon~D. Anderson.
\newblock {\em Modern compressible flow with historical perspective, 3tf ed}.
\newblock Mc Graw Hill, 2003.

\bibitem{batchelor}
George~Keith Batchelor.
\newblock {\em An Introduction to Fluids Dynamics}.
\newblock Cambridge University Press, 1967.

\bibitem{chemin}
Jean~Yves Chemin.
\newblock {\em Perfect Incompressible Fluids}.
\newblock Oxford University Press, 1998.

\bibitem{feireisl}
Eduard Feireisl.
\newblock {\em Dynamics of Viscous Compressible Fluids}.
\newblock Oxford University Press, 2004.

\bibitem{flett}
T.~M. Flett.
\newblock {\em Differential analysis: differentiation, differential equations
  and differential inequalities}.
\newblock Cambridge University Press, 1980.

\bibitem{glz}
Zhenhua Guo, Hai-Liang Li, and Zhouping Xin.
\newblock Lagrange structure and dynamics for solutions to the spherically
  symmetric compressible navier-stokes equations.
\newblock {\em Comm. Math. Phys.}, 309 (2):371--412, 2012.

\bibitem{hoff-95}
David Hoff.
\newblock Global solutions of the navier-stokes equations for multidimensional
  compressible flow with discontinuous initial data.
\newblock {\em Journal of Differential Equations}, 120:215--254, 1995.

\bibitem{hoff-heatconducting}
David Hoff.
\newblock Discontinuous solutions of the navier-stokes equations for a
  multidimensional flows of heat-conducting fluids.
\newblock {\em Arch. Rational Mech. Anal}, 139:303--354, 1997.

\bibitem{hoff-2dim}
David Hoff.
\newblock Dynamics of singularity surfaces for compressible, viscous flows in
  two space dimensions.
\newblock {\em Communicatuions on Pure an Aplied Mathematics}, 55:1365--1407,
  2002.

\bibitem{hoff-santos}
David Hoff and Marcelo~M. Santos.
\newblock Lagrangean structure and propagation of singularities in
  multidimensional compressible flows.
\newblock {\em Arch. Rational Mech. Anal}, 188:509--543, 2008.

\bibitem{zhouping}
Xiangdi Huang, Jing Li, and Zhouping Xin.
\newblock Global well-posedness of classical solutions with large oscillations
  and vacuum to the three-dimensional isentropic compressible navier-stokes
  equations.
\newblock {\em Comm. Pure Appl. Math.}, 65 (4):549--585, 2012.

\bibitem{teixeira}
Edson~J. Teixeira and Marcelo~M. Santos.
\newblock Lagrangian structure for compressible flow in the half-space with the
  navier boundary condition.
\newblock {\em arXiv:1805.00052}, 2018.

\bibitem{zhang}
Ting Zhang and Daoyuan Fang.
\newblock Compressible flows with a density-dependent viscosity coefficient.
\newblock {\em Siam Journal of Mathematical Analisys}, 41:2453--2488, 2010.

\end{thebibliography}

\end{document}